\documentclass[11pt,reqno]{amsart}
\usepackage{amsfonts,amssymb,amsmath,amsthm}
\usepackage[margin=1in]{geometry}
\usepackage[hidelinks]{hyperref}
\usepackage{enumerate}
\usepackage{microtype}
\usepackage{mathtools}
\usepackage{cite}
\usepackage{color}
\usepackage{tikz-cd}
\usepackage{cleveref}
\usepackage{enumitem}
\usepackage{microtype}
\usepackage{mathtools}
\usepackage{color}
\usepackage{tikz-cd}
\usepackage{comment}
\usepackage{graphicx}
\usepackage{bbm, color}
\usepackage{latexsym,mathrsfs,bm}
\usepackage{amssymb}
\usepackage{amsmath}
\usepackage{mathtools}
\usepackage{amsthm}
\usepackage{mathrsfs}
\usepackage{geometry}
\usepackage{commath}
\usepackage{hyperref}
\usepackage{multirow}
\usepackage{xcolor}
\usepackage{breqn}
\usepackage{tikz}
\usetikzlibrary{matrix}
\usepackage{ulem}

\input colordvi

\newtheorem{thm}[equation]{Theorem}
\newtheorem{cor}[equation]{Corollary}
\newtheorem{lem}[equation]{Lemma}

\newtheorem{rem}[equation]{Remark}
\newtheorem{prop}[equation]{Proposition}
\usepackage{cite}

\newcommand{\Nb}{\mathbb{N}}

\newcommand{\dd}{\textnormal{d}}

\newcommand{\R}{\mathbb{R}}
\newcommand{\Q}{\mathbb{Q}}
\newcommand{\C}{\mathbb{C}}

\newcommand{\Z}{\mathbb{Z}}

\newcommand{\PP}{\mathbb{P}}

\newcommand{\bigX}{\mathbf{X}}
\newcommand{\bx}{\mathbf{x}}

\newcommand{\balpha}{\boldsymbol{\alpha}}

\newcommand{\by}{\mathbf{y}}
\newcommand{\bu}{\mathbf{u}}
\newcommand{\bv}{\mathbf{v}}
\newcommand{\bigV}{\mathbf{V}}

\newcommand{\bz}{\mathbf{z}}
\newcommand{\wbx}{\langle \mathbf{x} \rangle}

\newcommand{\ang}[1]{ \langle #1 \rangle}

\newcommand{\bbeta}{\boldsymbol{\beta}}

\newcommand{\bt}{\mathbf{t}}

\def\MR#1{}

\newcommand{\br}[1]{\langle #1 \rangle}
\let\originalleft\left
\let\originalright\right
\renewcommand{\left}{\mathopen{}\mathclose\bgroup\originalleft}
\renewcommand{\right}{\aftergroup\egroup\originalright}

\numberwithin{equation}{section}

\title{Rational points on varieties defined by multihomogeneous diagonal forms}
\author{Doyon Kim} 
\address{Mathematical Institute of the University of Bonn, Bonn, Germany}
\email{kimdoyon@math.uni-bonn.de}
\author{Tian Wang}
\address{Department of Mathematics \&
Statistic, Concordia University, Montreal, Canada}
\email{tian.wang@concordia.ca}

\subjclass[2020]{Primary 11P55, 11D45, 11D72.}
\keywords{Hardy-Littlewood circle method, forms in many variables, hyperbola method}
\begin{document}

\maketitle 

\begin{abstract}
We give an asymptotic formula for the number of rational points of bounded height on algebraic varieties defined by systems of multihomogeneous diagonal equations. The proof uses the Hardy-Littlewood circle method and the hyperbola method developed by Blomer and Br\"udern. 
\end{abstract}

\section{Introduction}

 Motivated by Manin’s conjecture \cite{FrMaTs1989, BaMa1990, Pe1995}, we study the number of rational points of a bounded height satisfying systems of $R$ multihomogeneous diagonal equations of the form 
\begin{equation}\label{eq:equation}
 \sum\limits_{j=1}^s \lambda_{i,j} (x_{1,j}x_{2,j}\cdots x_{k,j})^d=0, \quad 1\leq i \leq R
 \end{equation}
for positive integers $d$ and $s$. This defines a variety $X_0$ over $\Q$ in the projective space 
\[
(\PP_\Q^{s-1})^k:=\underbrace{\PP_\Q^{s-1}\times \cdots \times \PP_\Q^{s-1}}_k
 \]
 with variables $\bx_i=(x_{i,1},\ldots,x_{i,s})$ for $1\leq i\leq k$. The $\Q$-rational points of $X_0$ are in $1$-to-$2^k$ correspondence with the solutions of \eqref{eq:equation} in primitive vectors $\bx_i\in \Z^{s}$, i.e., $\gcd(x_{i, 1}, \ldots, x_{i, s})=1$ for each $1\leq i\leq k$, because
 $\pm\bx_i$ represent the same point in $\PP_\Q^{s-1}$. For a rational point $(\bx_1, \ldots, \bx_k)\in (\PP^{s-1}_{\Q})^k$ represented by primitive vectors, we define the height as
\begin{equation}\label{eq: height}
 H((\bx_1, \ldots, \bx_k))=(\vert\bx_1\vert\vert\bx_2\vert\cdots \vert\bx_k\vert)^{s-Rd},
 \end{equation}
 where for $1\leq i\leq k$,
\begin{equation}\label{eq:abs-value}
\vert\bx_i\vert := \max_{1\leq j\leq s} \{\vert x_{i,j}\vert\}.
\end{equation}
For $B\geq 1$, we define
 \begin{equation}\label{eq: NB}
 N(B)=\#\{P=(\bx_1, \ldots, \bx_k)\in X_0(\Q): H(P)\leq B, \, x_{i, j}\neq 0 \; \text{for all } 1\leq i\leq k,\, 1\leq j\leq s \}. 
 \end{equation}
The main goal of this paper is to give an asymptotic formula for $N(B)$. \par
For $R=1$, Blomer and Br\"{u}dern developed a hyperbola method in \cite{BlBr2018} and proved an asymptotic formula for $N(B)$ with an explicit lower bound of $s$. Roughly speaking, the hyperbola method derives an asymptotic formula for sums over the hyperbolic region $\sum_{1\leq x_1x_2\cdots x_k\le B} h(x_1, \ldots, x_k)$, from asymptotic formulas for sums over the rectangular box $\sum_{1\leq x_{1}, x_{2},\ldots, x_r\le B} \tilde{h}(x_1, \ldots, x_r)$ with $1\le r\le k$, where $\tilde{h}$ is a specialization of $h$ obtained by fixing some of the variables. Such box-sum estimates are generally more accessible via the circle method (see Section \ref{sec:hypmethod} for further details). \par
The hyperbola method and its generalizations have many applications in the study of rational points on algebraic varieties, including toric and spherical varieties \cite{Sc14, Ho2024, BlBrDeGa2024, BlBr2017, Mi16, PiSc2024, BlBrSa2018}. We generalize the result of Blomer and Br\"udern to $R> 1$ by combining the hyperbola method and other classical methods developed by Br\"udern, Cook, Davenport, and Lewis \cite{Bruden1990,Cook1985, DaBook2005, DaLe1963,DaLe1966,DaLe1969}.
When studying systems of equations, we do not gain extra information if one equation is a linear combination of the others, so it is natural to impose certain rank conditions on the coefficient matrix
\begin{equation}
\label{eq:Lambda}
\Lambda =\left(\begin{array}{ccc}
 \lambda_{1,1} & \cdots &\lambda_{1,s} \\
 \vdots & \ddots & \vdots \\
 \lambda_{R,1} & \cdots & \lambda_{R,s}
\end{array}\right)
\end{equation}
associated to the equation \eqref{eq:equation}. For a matrix $M$ of size $r\times t$ over $\R$ and $0\leq l\leq r$, let $\mu(l,M)$ be the maximal number of columns from $M$ spanning a subspace of dimension $l$ in $\R^r$. Note that if $\mu(l,M)\leq m$, then any set of $m$ columns of $M$ contains at least $l$ linearly independent columns.
We introduce the following invariant associated with $\Lambda$. We define \begin{equation}\label{Kdef}
 K=K(\Lambda)=\max(\Delta,Rm_1,Rm_2),
\end{equation} where 
\[
\Delta=\max_{A\in \mathcal{M}} \abs{\det A}, \quad m_1=\max_{\substack{1\leq i\leq R\\ 1\leq j\leq s}}|\lambda_{i, j}|,\quad m_2=\max_{A\in \mathcal{M}} \max_{1\leq i,j\leq R} \bigl|\left(\operatorname{adj}(A)\right)_{i,j}\bigr|, \]
with $\mathcal{M}$ denoting the set of invertible $R\times R$ submatrices of $\Lambda$.

Our result also requires $s$ to be of certain size. To give an explicit bound on the size of $s$, we define $n_0=n_0(d)$ to be the smallest even natural number with the property that
\begin{equation}\label{threshold1}
 \int_0^1 \bigg\vert \sum_{1\leq x\leq P} e^{2\pi i \alpha x^d}\bigg\vert ^{n_0(d)} \dd\alpha \ll P^{n_0(d)-d+\epsilon}
\end{equation}
for any $\epsilon>0$. As noted after equation (1.3) in \cite{BlBr2018}, we have the lower bound $n_0\geq 2d$. It is conjectured that $n_0(d)=2d$, and the best known result is $n_0(d)\leq d^2-d+2\lfloor \sqrt{2d+2} \rfloor-1$ \cite[Corollary 14.7]{Wooley2019}. If $d$ is a smooth number, then there is a much better bound for $n_0(d)$ \cite{BrCo1992, RoYa2026}.

We first establish an asymptotic formula of box-sum type. 
For $\bigX\in [1, \infty)^k$, let $M_{\Lambda}(\bigX)$ denote the number of integral solutions of \eqref{eq:equation} subject to the box-sum constraints:
\begin{equation}\label{eq:boxbound}
 1\leq |x_{i,j}|\leq X_i, \quad 1\leq i\leq k, \quad 1\leq j\leq s.
\end{equation} 
Using the circle method and a combinatorial argument similar to \cite[Section 4.4]{BlBr2018}, we derive an asymptotic formula for $M_\Lambda(\bigX)$.
\begin{thm}\label{thm:boxsum}
Let $d,k,R$ be natural numbers, and let $s$ be a natural number satisfying $s\geq R(n_0+1)$. Assume that the coefficient matrix $\Lambda$ contains a submatrix $\Lambda^*$ of size $R\times R(n_0+1)$ such that $\mu(l,\Lambda^*)\leq (n_0+1) l$ for all $0\leq l\leq R-1$. Then there exist positive constants $\delta$ and $A$, and a non-negative constant $C_\Lambda$ such that
\begin{equation}\label{eq:boxsum} M_\Lambda(\bigX)=C_\Lambda \br{\bigX}^{s-Rd}+O\bigl(K^{A} \br{\bigX}^{s-Rd} (\min\limits_{1\leq j\leq k} X_j)^{-\delta}\bigr),\end{equation}
 where $\br{\bigX}=X_1\cdots X_k$, and $K$ is the value associated to $\Lambda$ defined by \eqref{Kdef}. Furthermore, $C_\Lambda>0$
if and only if \eqref{eq:equation} has solutions in $\R$ and in $\Q_p$ for every prime $p$, where $x_{i,j}\neq 0$ for all $1\leq i\leq k$, $1\leq j\leq s$.
\end{thm} 
The constant $C_\Lambda$ is explicitly defined in terms of singular series and singular integrals in \eqref{eq:boxsumconst}, and the number $A$ is also given explicitly in \eqref{Adef}. As stated in the first paragraph of this section, we need to count the number of primitive integral solutions to estimate $N(B)$. We obtain the asymptotic formula of the primitive solutions in a box directly from \eqref{eq:boxsum} for all solutions in Section~\ref{sec:priasymp}. After that, by a direct application of the hyperbola method, we convert the asymptotic formula into a sum of hyperbolic type.
\begin{thm}\label{thm:hypsum}
Let $d,k,R,s$, and $\Lambda$ satisfy the hypotheses in Theorem~\ref{thm:boxsum}. Then there exist a positive constant $\delta$, a nonnegative constant $C$, and a monic polynomial $Q\in\R[X]$ of degree $k-1$ such that
\[N(B)=CBQ(\log B)+O(B^{1-\delta}).\]
Moreover, the number $C>0$ if and only if \eqref{eq:equation} has solutions in $\R$ and in $\Q_p$ for every prime $p$, where $x_{i,j}\neq 0$ for all $1\leq i\leq k$, $1\leq j\leq s$.
\end{thm}

\begin{rem}
If $R=1$, it is known that the equation \eqref{eq:equation} has $p$-adic solutions for every prime $p$ if $s>kd^2$ \cite{DaLe1963}. 
\end{rem}
\begin{rem}\label{rem:uniformity}
For fixed $d, k, R$, and $s\geq R(n_0+1)$, note that the implied constant in the big $O$ notation of \eqref{eq:boxsum} is uniform in $K(\Lambda)$. Using this, we can also derive a uniformity statement for the asymptotic formula for $N(B)$: fix $K_0>0$, and consider a family of systems of equations whose coefficient matrices $\Lambda$ satisfy $K(\Lambda)\leq K_0$. Then the asymptotic formula for $N(B)$ holds uniformly across the family, that is, the implied constant in the error term is independent of the choice of $\Lambda$. A proof of this statement is given in Section~\ref{sec:uniformity}.
\end{rem}

\begin{rem}
We could also ask whether an analogue of Theorem~\ref{thm:hypsum} 
can be obtained for systems of multihomogeneous diagonal forms of differing degrees. More precisely, consider a system of equations of the form 
\[
\sum_{j=1}^s\lambda_{i, j}(x_{1, j}^{d_1}\cdots x_{k, j}^{d_k})=0, \quad 1\leq i\leq R,
\]
where $d_{i}, k, s$ are natural numbers.  This defines a variety over $\Q$ in a product of weighted projective spaces 
 with variables $\bx_i=(x_{i,1},\ldots,x_{i,s})$ for $1\leq i\leq k$. Now,  for a rational point $(\bx_1, \ldots, \bx_k)$ represented by primitive vectors, we need to modify the definition of height in \eqref{eq: height} as
\begin{equation*}
 H((\bx_1, \ldots, \bx_k))=\vert\bx_1\vert^{s-d_1}\vert\bx_2\vert^{s-d_2}\cdots \vert\bx_k\vert^{s-d_k},
 \end{equation*}
 where $\vert\bx_i\vert$ is defined in \eqref{eq:abs-value}. When $R=1$, the asymptotic behaviour of the number of rational points of bounded height was established in \cite{Zhao2024}. The methods used in this paper to generalize to systems of equations also apply in this setting, and we leave this direction for future work.

\end{rem}

We conclude the introduction with a brief discussion of the geometry underlying the asymptotic formula given in \Cref{thm:hypsum}. 
A computation of the Jacobian matrix of $X_0$ shows that $X_0$ is a normal (but not smooth) variety under the assumptions of \Cref{thm:hypsum}. Suppose further that $X_0$ is a complete intersection in $(\PP_\Q^{s-1})^k$. Then the anticanonical line bundle on the smooth locus $X_0'$ of $X_0$ is given by $\mathcal{O}(\underbrace{s-dR, \ldots, s-dR}_{k})\mid_{X'_0}$. In particular, since $s>dR$, the anticanonical bundle is very ample and the corresponding height function coincides with the one defined in \eqref{eq: height}. 

According to Manin's conjecture, the number $N(B)$ of rational points of anticanonical height at most $B$  on a Fano variety with Picard rank $r$ satisfies the asymptotic formula \cite{FrMaTs1989, BaMa1990, Pe1995}
\[
N(B)\sim C_{\text{Peyre}} B(\log B)^{r-1}. 
\]
This closely resembles the result in \Cref{thm:hypsum}, despite the fact that $X_0$ is not smooth.  The degree of the polynomial $Q(X)$ in \Cref{thm:hypsum} reflects the Picard rank of the ambient space $(\PP_\Q^{s-1})^k$.
 
We now offer a comparison with a related result of Schindler \cite[Theorem 1.1]{Sc2016}. The equations of the form \eqref{eq:equation} can also be viewed as bihomogeneous equations on the biprojective space $\PP_\Q^{k_1(s-1)}\times \PP_\Q^{k_2(s-1)}$ with $k_1, k_2\geq 1$ and $k_1+k_2=k$. Although the defining equations are the same, the height function used in \cite{Sc2016} differs from the one defined in \eqref{eq: height}. Consequently, Schindler obtained an asymptotic of the form $C_1B\log B$
for the number of rational points on a Zariski open subset $U\subseteq X_0$, assuming $s \gg R^3 2^{dk} d^2k_1k_2/k$. The exponent in $\log B$ reflects the Picard rank of the ambient space $\PP_\Q^{k_1(s-1)}\times \PP_\Q^{k_2(s-1)}$. In contrast, \Cref{thm:hypsum} requires only the weaker assumption $s\gg R d^2$ on $s$. 

It is also noteworthy that the variety $X_0$ admits a natural description as intersections of toric hypersurfaces. For background and further details on such embeddings, see \cite{CocLittleSchenck2011, Brion1997}. Manin's conjecture is established for toric varieties \cite{BaTs1995, BaTs1996, BaTs1998}, and extended to certain smooth subvarieties, as in the works \cite{Mi16, Sc2016, PiSc2024}.
\subsection*{Notation and convention}
Boldface letters denote vectors. For $\bx=(x_{1},x_2,\ldots,x_k)\in \R^k$, we let $\wbx=x_1x_2\cdots x_k$ and $\d\dd\boldsymbol{\bx}=\dd x_1\cdots \dd x_k$. We use superscript $\mathrm{T}$ to denote transpose, so $\bx^{\mathrm{T}}$ is the $k$-dimensional column vector that corresponds to $\bx$. For the two vectors
\[\bx=(x_{1},x_2,\ldots,x_k), \quad \by=(y_{1},y_2,\ldots,y_l),\]
we let $(\bx,\by)$ denote the vector of dimension $k+l$,
\[(\bx,\by)=(x_1,\ldots,x_k,y_1,\ldots,y_l).\]
For $\bu,\bv\in \R^k$, we write $\bu \leq \bv$ if $u_i\leq v_i$ for all $1\leq i\leq k$. Accordingly, for $\mathbf{X}=(X_1,\ldots,X_k)\in[1,\infty)^k$, we write
$\sum_{\bx\le\bigX}=\sum_{\substack{1\leq x_j\leq X_j\\1\le j\leq k}}$.

For a natural number $n$, let $\tau(n)$ denote the number of divisors of $n$, $\varphi(n)$ the Euler's totient function, and $\mu(n)$ the M\"obius function. For integers $a_1,\ldots, a_n$, we let $(a_1;a_2;\ldots;a_n)$ denote the greatest common divisor of $a_1,\ldots, a_n$. For $\alpha\in\C$, we put
$e(\alpha)=e^{2\pi i \alpha}$, and let $\zeta(\alpha)$ denote the Riemann zeta function.

Let $X$ be a set and $f: X\to \R$ and $g: X\to \R_{\geq 0}$ be two functions. We write $f\ll g$ if there is an absolute constant $C>0$ such that $|f(x)|\leq C g(x)$ for all $x\in X$. We also write $f=O(g)$ interchangeably. 
Throughout this paper, the parameters $d, k, R, s$ are considered fixed, and the implied constants in $\ll$ or big $O$ notation may depend on them.
\section*{Acknowledgements}
The first author was supported by ERC Advanced Grant 101054336 and Germany’s 
Excellence Strategy grant EXC-2047/1 - 390685813. We thank Valentin Blomer, Nick Rome, and Igor Shparlinski for helpful comments. The second author would like to thank the Max Planck Institute for Mathematics for its funding and stimulating atmosphere of research. We would also like to thank the referees for their feedback.

\section{Preliminaries}\label{sec:prelim}

\subsection{Hardy-Littlewood circle method}\label{sec:circlemethod}
Let $M^+_\Lambda(\bigX)$ denote the number of integral solutions of \eqref{eq:equation} satisfying \eqref{eq:boxbound} with all $x_{i,j}$ positive. To prove Theorem~\ref{thm:boxsum}, we first obtain the asymptotic formula for $M^+_\Lambda(\bigX)$ using the circle method, and then deduce the asymptotic formula for $M_\Lambda(\bigX)$ from it by simple combinatorial arguments. 

Let \[f(\alpha)=f_{k, d}(\alpha,\bigX)=\sum\limits_{\bx\leq \bigX} e(\alpha \wbx^d ).\]
Let $\bigX=(X_1,\ldots,X_k)\in [1,\infty)^k$. For $1\leq j\leq s$, write $\bx_j=(x_{1,j},\ldots,x_{k,j})\in \mathbb{N}^k$. We have
\begin{equation}\label{eq:ML+int}
\begin{aligned}
M^{+}_\Lambda(\bigX)&=\int\limits_{[0,1]^R} \sum_{\substack{\bx_j\leq \bigX \\ 1\leq j\leq s}}e\biggl( \alpha_1 \Bigl( \sum_{j=1}^s \lambda_{1,j} \br {\bx_j}^d \Bigr)+\cdots+ \alpha_R\Bigl( \sum_{j=1}^s \lambda_{R,j} \br {\bx_j}^d \Bigr)\biggr)\,\dd\boldsymbol{\alpha} \\
&=\int\limits_{[0,1]^R} \sum_{\bx_1\leq \bigX} e\Bigl(L_1(\boldsymbol{\alpha})\br{\bx_1}^d\Bigr)\cdots \sum_{\bx_s\leq \bigX} e\Bigl(L_s(\boldsymbol{\alpha}) \br{\bx_s}^d\Bigr)\,\dd\boldsymbol{\alpha} \\
&=\int\limits_{[0,1]^R} f\bigl(L_1(\boldsymbol{\alpha})\bigr)f\bigl(L_2(\boldsymbol{\alpha})\bigr)\cdots f\bigl(L_s(\boldsymbol{\alpha})\bigr)\,\dd\boldsymbol{\alpha},
\end{aligned}
\end{equation}
where $\boldsymbol{\alpha}=(\alpha_1,\ldots,\alpha_R)$, $\dd\boldsymbol{\alpha}=\dd\alpha_1\cdots \dd\alpha_R$, and $L_j$, $1\leq j\leq s$ is the linear form on $\R^R$ defined by
\begin{equation}\label{eq:Ljdef} 
L_j(\mathbf{v})=\lambda_{1,j}v_1+\cdots+\lambda_{R,j}v_R=\sum_{i=1}^R \lambda_{i,j}v_i.\end{equation}
Let $F(\boldsymbol{\alpha})=f\bigl(L_1(\boldsymbol{\alpha})\bigr)f\bigl(L_2(\boldsymbol{\alpha})\bigr)\cdots f\bigl(L_s(\boldsymbol{\alpha})\bigr)$. Observe that, by the symmetry of \eqref{eq:equation}, permuting the coordinates $X_i$ of $\bigX=(X_1,\ldots,X_k)$ does not affect the number of integral solutions. Hence, without loss of generality, we may assume that $X_1\geq\cdots\geq X_k$, and write
\begin{equation}\label{eq:P=X}
 P=\br{\bigX}.
\end{equation}
To apply the circle method, we divide the region $[0,1]^R$ into major and minor arcs. We use the following notation: For $Q$ in the range $1\leq Q\leq P^{1/(dk)}$, we let $\mathfrak{N}(Q)$ denote the pairwise disjoint intervals $\abs{q\alpha-a}\leq QP^{-d}$ with $a$ and $q$ subject to $1\leq q\leq Q$, $a\in \Z$ and $(a;q)=1$, and let $\mathfrak{n}(Q)=\R\setminus \mathfrak{N}(Q)$. We define 
\[\mathfrak{M}(Q)=\mathfrak{N}(Q)\cap [0,1], \quad \mathfrak{m}(Q)=\mathfrak{n}(Q)\cap [0,1].\]
Let
\begin{equation}\label{eq:omegaRUdef}
\omega_R=R^{-1}(8dk)^{-8}, \quad U=X_k^{k\omega_R/s}.\end{equation}
The following is \cite[Lemma 3.7]{BlBr2018} with $\omega_R$ in place of $\omega=\omega_1$. The proof of \cite[Lemma 3.7]{BlBr2018} uses Lemmas~3.5 and~3.6 of that paper. In Lemma~3.5, they establish a Weyl-type inequality for $f(\alpha)$ via Weyl's differencing argument, H\"{o}lder's inequality, and  induction on the length $k$ of $\bigX$. Lemma~3.6 establishes a power-saving bound for $f(\alpha)$ on the minor arcs $\mathfrak{n}(P^{\omega})$ using the Weyl-type inequality together with Dirichlet's approximation theorem. Lemma~3.7 refines the estimate by extending it to the wider minor arcs $\mathfrak{n}(X_k^{k\omega/s})$. In this paper we use $\omega_R$ in place of $\omega$, so the refinement corresponds to the minor arcs $\mathfrak{n}(U)$. Such a modification is necessary as we can see from \eqref{eq:defineQ} in the proof of \Cref{lem:minorcover}.
\begin{lem}\label{lem:wminorarcs} We have
\[\sup_{\alpha\in \mathfrak{n}(U)} \abs{f(\alpha)}\ll P^{1-t} + PU^{-1/(5dk)},\]
where $t=\frac{\omega_R}{2^{(d-1)k+1} dk}$.
\end{lem}
\begin{proof}
The definition of $f(\alpha)$ in this paper is the same as that in \cite{BlBr2018}, so Lemma~3.5 in that paper applies without modification. Also, in the proof of \cite[Lemma 3.6]{BlBr2018}, the particular choice of $\omega$ does not matter, as long as $\omega$ is sufficiently small in terms of $d$ and $k$. Therefore the same proof applies after replacing $\omega$ by $\omega_R$. The same reasoning applies to \cite[Lemma 3.7]{BlBr2018}.
\end{proof}
Let 
\begin{equation}\label{eq:Wdef}W=KU^s.\end{equation}
For $1\leq q\leq W$ and $\mathbf{A}=(A_1,\ldots,A_R)$ such that $0\leq A_1,\ldots, A_R\leq q$, $(A_1;\ldots;A_R;q)=1$, define
\[\mathfrak{K}(\mathbf{A},q)=\Bigl\{\boldsymbol{\alpha}=(\alpha_1,\ldots,\alpha_R)\in [0,1]^R : \,\Big\lvert\alpha_j-\frac{A_j}{q}\Big\rvert\leq WP^{-d} \text{ for all } 1\leq j\leq R\Bigr\}.\]
Define the major arcs by
\begin{equation}\label{def:majorarc}\mathfrak{K}=\bigcup_{1 \leq q\leq W} \bigcup_{\mathbf{A}} \mathfrak{K}(\mathbf{A},q),\end{equation}
where the second union is over all $0\leq A_1,\ldots, A_R \leq q$, $(A_1;\ldots;A_R,q)=1$.
Observe that \[2K^3U^{3s}=2K^3X_k^{3k\omega_R}\leq 2K^3 P^{3\omega_R}\leq 2K^3 P^{\frac{3}{2^{24}}}.\]
Thus, if $P$ is large enough so that \begin{equation}\label{eq:Psize} 2K^3<P^{d-\frac{3}{2^{24}}},\end{equation} then $2K^3U^{3s}<P^{d}$. It follows that for any $1\leq q_1, q_2\leq W$ we have
\[\Big\lvert\frac{A_1}{q_1}-\frac{A_2}{q_2}\Big\rvert\geq \frac{1}{q_1q_2}\geq \frac{1}{K^2U^{2s}}> 2KU^{s}P^{-d}\]
whenever $\frac{A_1}{q_1}\neq \frac{A_2}{q_2}$. Therefore, for such $P$, the distance between the centers of any two boxes in $\mathfrak{K}$ is greater than their length, hence $\mathfrak{K}$ is a disjoint union. Let $\mathfrak{k}=[0,1]^R\setminus \mathfrak{K}$ denote the minor arcs. For $\mathfrak{a}\subset [0, 1]^R$, let
\begin{equation}\label{def:IKk}I(\mathfrak{a})=\int_{\mathfrak{a}} F(\balpha) \, \dd\balpha.\end{equation}
We have $M_{\Lambda}^{+}(\bigX)=I(\mathfrak{K})+I(\mathfrak{k})$. Our goal in using the circle method is to show that the integrals $I(\mathfrak{K})$ and $I(\mathfrak{k})$ contribute to the main term and the error term of the asymptotic formula for $M_{\Lambda}^{+}(\bigX)$, respectively.
The following mean-value estimate is crucial in the process.
\begin{lem}[Lemma 3.8 of \cite{BlBr2018}]\label{lem:sigmacancel} Fix a real number $\sigma>n_0(d)$. Then, for $\bigX \in [1,\infty]^k$, one has
\[\int_{0}^{1} \abs{f(\alpha)}^\sigma \textnormal{d} \alpha \ll \br{\bigX}^{\sigma-d}.\] 
\end{lem}

\subsection{Hyperbola method}\label{sec:hypmethod}
In this section, we record the transition theorem of Blomer and Br\"udern \cite[Section 2]{BlBr2018} from box sums to hyperbolic sums, and give a brief summary of their hyperbola method. For an arithmetic function $h:\mathbb{N}^k\to \C$ we denote the ``hyperbolic sum'' by
\begin{equation}\label{eq:upsilon}\Upsilon(N) = \sum_{u_1 u_2 \cdots u_k \leq N} h(\bu),\end{equation}
and denote the ``box sum'' by
\[\Theta(X_1, \dots, X_k) = \sum\limits_{\substack{1 \leq x_i \leq X_i \\ 1 \leq i \leq k}} h(\bx).\]
Fix positive numbers $\alpha, c, \delta $ with $\delta<\min(1,\alpha)$, and let $\nu,D$ be real numbers that satisfy $0<\nu\leq 1$ and $D\geq 0$. A collection $\mathscr{H}$ of arithmetic functions $h:\mathbb{N}^k\to \C$ is called an $(\alpha, c, D, \nu, \delta)$-family if the following three conditions are satisfied:
\begin{enumerate}
 \item [(I)] For any $ h \in \mathscr{H} $ there is a real number $c_h \in [0, c] $ such that the asymptotic formula
\[
\sum_{\mathbf{x} \leq \bigX} h(\mathbf{x}) = c_h \langle \bigX \rangle ^\alpha + O\Bigl( \langle \bigX \rangle^\alpha \Bigl( \min_{1 \leq i \leq k} X_i \Bigr)^{-\delta} \Bigr)
\]
holds uniformly in $ h \in \mathscr{H} $ and $ X_i \geq 1 $ ($ 1 \leq j \leq k $).
\item[(II)] For $h \in \mathscr{H}$ and $r \in \mathbb{N}$ with $1 \leq r \leq k - 1$, there exists an arithmetic function $c_{h,r}: \mathbb{N}^r \to [0, \infty)$ such that for any $\mathbf{u} \in \mathbb{N}^r$ and $\mathbf{V}=(V_{r+1},\ldots,V_{k})\in [1, \infty)^{k-r}$, the asymptotic formula
\[
\sum_{\mathbf{v} \leq \mathbf{V}} h(\mathbf{u}, \mathbf{v}) = c_{h,r}(\mathbf{u}) \langle \mathbf{V} \rangle^\alpha + O\Bigl( \langle \mathbf{V} \rangle^\alpha \abs{\mathbf{u}}^D (\min_{r+1\leq i\leq k} V_i)^{-\delta} \Bigr) 
\]
holds uniformly for $h \in \mathscr{H}$, $V_i \geq 1$ ($r+1\leq i\leq k$) and $\abs{\mathbf{u}} \leq \langle \mathbf{V} \rangle^\nu$.
\item[(III)] For all $h \in \mathscr{H}$ and $\sigma \in S_k$, the function $h_\sigma$ defined by $h_\sigma(\mathbf{x})=h\bigl((x_{\sigma(1)},\ldots, x_{\sigma(k)})\bigr)$ is in $\mathscr{H}$.
\end{enumerate} 
The following theorem gives the transition from box sums to hyperbolic sums.
\begin{thm}[Theorem 2.1 of \cite{BlBr2018}]\label{thm:BlBrhyp}
Let $ k \geq 2 $, and let $ \mathscr{H} $ be an $ (\alpha, c, D, \nu, \delta) $-family of arithmetic functions $ h: \mathbb{N}^k \to [0, \infty) $. For any $ h \in \mathscr{H} $, let $ \Upsilon(N) $ be defined by \eqref{eq:upsilon}. There exists a positive number $ \eta $ with the property that for any $ h \in \mathscr{H} $ there is a polynomial $ P_h \in \mathbb{R}[x] $ of degree at most $k-2$ such that the asymptotic formula
 \[\Upsilon(N) = N^\alpha \left( \frac{c_h \alpha^{k-1}}{(k-1)!} (\log N)^{k-1} + P_h(\log N) \right) + O(N^{\alpha - \eta})\]
holds uniformly in $h\in \mathscr{H}$.
\end{thm}
To prove this theorem, Blomer and Br\"udern decompose the hyperbolic region into three parts. First, in the region away from the spikes, where all $u_j$ are large, the contribution to the sum $\Upsilon(N)$ is evaluated based on condition (I). Secondly, the contribution from the region where all $u_j$ are small is not significant. Finally, in the intermediate region, where some coordinates of $\bu$ are large and others small, conditions (II) and (III) are used to control the contribution.
\section{Treatment of the minor arcs}\label{sec:minor}
In this section, we estimate the contribution from the minor arcs. Assume that $s\geq R(n_0+1)$, and that the coefficient matrix $\Lambda$ contains a submatrix $\Lambda^*$ of size $R\times R(n_0+1)$ such that $\mu(l,\Lambda^*)\leq (n_0+1) l$ for all $0\leq l\leq R-1$. To simplify the integral \eqref{eq:ML+int} which involves linear form, we apply a linear change of variables. The rank condition guarantees that such change of variables is invertible, by the following combinatorial result.
\begin{lem}[Aigner's criterion] \label{lem:Aigner} An $r\times nr$ matrix $M$ can be grouped into $n$ blocks of $r\times r$ invertible matrix if and only if $M$ satisfies $\mu(d,M)\leq nd$ for all $0\leq d\leq r$.
\end{lem}
\noindent For a proof see \cite[Proposition 6.45]{Aig1979} or \cite[Lemma 2]{LPW1988}. \par 
For $1\leq j\leq s$, let $C_j$ denote the $j$-th column of $\Lambda$. Observe that by the symmetry of the equation \eqref{eq:equation}, permuting the columns of $\Lambda$ does not affect the number of integral solutions. Thus, for the rest of the paper, we may assume that $\Lambda^*$ takes the first through the $R(n_0+1)$-th columns of $\Lambda$, and that for any $0\leq t\leq n_0$ the matrix $[C_{tR+1}, \ldots ,C_{(t+1)R}]$ is invertible.

Recall the notation introduced in \Cref{sec:circlemethod}. For each $1\leq j\leq s$, define the set
 \[\mathfrak{l}_j=\{\boldsymbol{\alpha}=(\alpha_1,\alpha_2,\ldots,\alpha_R)\in [0,1]^R:L_j(\boldsymbol{\alpha})\in \mathfrak{n} (U)\}.\]

\begin{lem}\label{lem:minorcover}
If $\balpha=(\alpha_1,\ldots,\alpha_R)\in [0,1]^R$ is not in $\mathfrak{l}_j$ for all $1\leq j\leq R$, then $\balpha\in \mathfrak{K}$.
\end{lem}
\begin{proof}
Suppose that $\balpha\notin \mathfrak{l}_j$ for each $1\leq j\leq R$. Then for each $1\leq j\leq R$, there exists $q_j\in \mathbb{N}$, $a_j\in \Z$ such that $1\leq q_j\leq U$ and
\[\Bigl\vert L_j(\balpha)-\frac{a_j}{q_j}\Bigr\vert
\leq \frac{U}{q_j P^{d}}.\]
Let $\gamma_j=L_j(\balpha)-a_j/q_j$. By the assumption above, $\Lambda_1=(\lambda_{i,j})_{1\leq i,j\leq R}$ is invertible. Let $\Delta_1=\abs{\det \Lambda_1}$, and let $\mu_{i,j}$, $1\leq i,j\leq R$ be the integers satisfying
\[\Lambda_1^{-1}=\frac{1}{\Delta_1}(\mu_{i,j})_{1\leq i,j\leq R}.\]
Let $K$, $m_1$, and $m_2$ be the values associated to $\Lambda$ by \eqref{Kdef}. Observe that $\Delta_1\leq \Delta$ and $|\mu_{i,j}|\leq m_2$, $1\leq i,j\leq R$. By \eqref{eq:Ljdef}, we have
\[\alpha_j=\frac{1}{\Delta_1}\sum_{i=1}^R \mu_{i,j} L_i(\balpha)=\frac{1}{\Delta_1}\sum_{i=1}^R \mu_{i,j}\frac{a_i}{q_i}+\frac{1}{\Delta_1}\sum_{i=1}^R \mu_{i,j}\gamma_i.\]
Let $t_j=\frac{1}{\Delta_1}\sum\limits_{1\leq i\leq R} \mu_{i,j}\frac{a_i}{q_i}$. We can write $\alpha_j=\frac{A_j}{Q}+\beta_j$ with $A_j\geq 0$, $Q>0$, and $(Q; A_1;\ldots;A_R)=1$, by taking
\[
\frac{A_j}{Q}=\begin{cases}
 0 & \text{if } t_j\leq 0 \\
 1 & \text{if } t_j\geq 1 \\
 t_j & \text{otherwise.}
\end{cases}\]
Observe that $Q$ satisfies 
\begin{equation}\label{eq:defineQ}
Q\leq \Delta_1 q_1\cdots q_R\leq \Delta U^R\leq KU^s=W,
\end{equation}
and that 
\[\abs{\beta_j}\leq \abs{\alpha_j-t_j}=\Bigl\vert\frac{1}{\Delta_1}\sum\limits_{1\leq i\leq R} \mu_{i,j}\gamma_i\Bigr\vert\leq KUP^{-d}\leq WP^{-d}.\]
By the definition \eqref{def:majorarc} of $\mathfrak{K}$, we conclude that $\balpha\in\mathfrak{K}$.
\end{proof}

\begin{lem}\label{lem:minorcoverint}
 There exists $\delta>0$, depending only on $d$, $k$, $s$, and $R$ such that for all $1\leq j\leq R$ we have
 \[\int_{\mathfrak{l}_j} \lvert F(\boldsymbol{\alpha})\rvert \,\dd\boldsymbol{\alpha}\ll K^{R} P^{s-Rd}X_k^{-\delta}.\]
\end{lem}
\begin{proof}
We may assume that $j=1$ since the proof for $2\leq j\leq R$ are similar. By Lemma~\ref{lem:wminorarcs}, there is a $\delta>0$ depending only on $d$, $k$, $s$, and $R$ that satisfies
 \begin{equation}\label{LjsupR} \sup\limits_{\balpha\in \mathfrak{l}_1} \abs{f(L_1(\balpha))}\leq \sup\limits_{\gamma\in \mathfrak{n}(U)} \abs{f(\gamma)} \ll PX_k^{-R\delta}.\end{equation}
For $1\leq t\leq n_0+1$, let
 \[F_t(\balpha)=\prod_{j=1}^R f\bigl(L_{(t-1)R+j}(\boldsymbol{\alpha})\bigr).\]
 Observe that the linear change of variables $L_{(t-1)R+j}(\balpha)=u_j$ for $1\leq j\leq R$ gives
 \begin{equation}\label{LCoV}\int_{[0,1]^R} \abs{F_t(\balpha)}^\sigma \dd\balpha \ll K^{R} \Bigl(\int_{[0,1]} \abs{f(u)}^\sigma du \Bigr)^R\end{equation}
 for any $1\leq t\leq n_0+1$ and $\sigma>0$. 
 By the trivial bound $f(\alpha)\ll P$, we have

 \begin{equation}\label{LjsupRrest}
 \begin{aligned}\int_{\mathfrak{l}_1} \abs{ F(\boldsymbol{\alpha})} \dd \balpha &\ll \int_{\mathfrak{l}_1}\abs{F_1(\balpha)F_2(\balpha)\cdots F_{n_0+1}(\balpha)}\dd\balpha\cdot\max_{\balpha\in\mathfrak{l}_1}\biggl\vert\prod_{R(n_0+1)<j\leq s }f\bigl(L_j(\balpha)\bigr)\biggr\vert \\
 &\ll P^{s-R(n_0+1)} \int_{\mathfrak{l}_1}\abs{F_1(\balpha)F_2(\balpha)\cdots F_{n_0+1}(\balpha)}\dd\balpha.
 \end{aligned}\end{equation}
 Since $\balpha\in\mathfrak{l}_1$, by the trivial bound $f(\alpha)\ll P$ and \eqref{LjsupR}, we have
 \[ \sup_{\balpha\in\mathfrak{l}_1} \abs{F_1(\balpha)}=\sup_{\balpha\in\mathfrak{l}_1}\abs{f(L_1(\balpha))f(L_2(\balpha))\cdots f(L_R(\balpha))} \ll P^R X_k ^{-R\delta}.\]
 Let $\sigma=n_0+\frac{1}{R}$. We have
 \[\frac{1}{R\sigma}+\frac{n_0}{\sigma}=1.\]
 Thus, by H\"{o}lder's inequality, Lemma~\ref{lem:sigmacancel}, and \eqref{LCoV}, we have
 \[\begin{aligned} \int_{\mathfrak{l}_1}\abs{F_1(\balpha)F_2(\balpha)\cdots F_{n_0+1}(\balpha)}\dd\balpha&\ll
 \Bigl(\int_{\mathfrak{l}_1} \abs{F_1(\balpha)}^{R\sigma}\dd\balpha\Bigr)^{\frac{1}{R\sigma}}\prod_{t=2}^{n_0+1} \Bigl(\int_{\mathfrak{l}_1} \abs{F_t(\balpha)}^{\sigma}\dd\balpha\Bigr)^{\frac{1}{\sigma}} \\
 &\ll PX_k^{-\delta} \Bigl(\int_{[0,1]^R} \abs{F_1(\balpha)}^{R\sigma-\sigma}\dd\balpha\Bigr)^{\frac{1}{R\sigma}}\prod_{t=2}^{n_0+1} \Bigl(\int_{[0,1]^R} \abs{F_t(\balpha)}^{\sigma}\dd\balpha\Bigr)^{\frac{1}{\sigma}} \\
 &\ll K^{R} P^M X_k^{-\delta},
 \end{aligned}\]
 where \[M=1+\frac{R\sigma-\sigma-d}{\sigma}+\frac{Rn_0(\sigma-d)}{\sigma}=R(n_0+1)-(Rn_0+1)\frac{d}{\sigma}=R(n_0+1)-Rd.\]
 Together with \eqref{LjsupRrest}, this gives the desired bound. 
\end{proof}
\begin{cor}\label{cor:Rminor}
There exists $\delta>0$, depending only on $d$, $k$, $s$, and $R$ such that
 \[I(\mathfrak{k})=\int_{\mathfrak{k}} F(\boldsymbol{\alpha})\,\dd\boldsymbol{\alpha}\ll K^{R} P^{s-Rd}X_k^{-\delta}.\]
\end{cor}
\begin{proof}
 By Lemma~\ref{lem:minorcover}, if $\balpha\in \mathfrak{k}$ then $\balpha\in \mathfrak{l}_j$ for some $1\leq j\leq R$. Thus the above bound follows immediately from Lemma~\ref{lem:minorcoverint}. 
\end{proof}

\section{Singular series and singular integrals}\label{sec:SS_and_SI}
In this section, we define the singular series and the singular integral associated to $\Lambda$, and establish their properties which are used later for the treatment of major arcs. Many of the notations we introduce here are from \cite{BlBr2018}. Recall the assumption for $\Lambda$ written after Lemma~\ref{lem:Aigner}.

For $q\in \mathbb{N}$ and $a\in \Z$, define 
\[
S_k(q,a)=\sum_{\substack{x_i=1 \\ 1\leq i\leq k}}^q e\Bigl(\frac{a\br{\bx}^d}{q}\Bigr),
\]
and
\[T_{\Lambda}(q)=q^{-ks}\sum_{\substack{0\leq A_1,\ldots,A_R\leq q \\ (A_1;\ldots;A_R;q)=1}} \prod_{j=1}^s S_k(q,L_j(\mathbf{A})).\]
We define the singular series as
\begin{equation}\label{Rsingser}\mathfrak{S}(\Lambda)=\sum_{q=1}^\infty T_{\Lambda}(q),\end{equation}
and its truncation as
\begin{equation}\label{Rsingsertrun}\mathfrak{S}(\Lambda,Y)=\sum_{q\leq Y} T_{\Lambda}(q).\end{equation}
For the singular integral, let
\begin{equation}\label{eq:vk} v(\beta)=v_k(\beta,\bigX)=\int_0^{X_k}\cdots\int_{0}^{X_1} e(\beta\br{\textbf{t}}^d)\, \dd t_1\cdots \dd t_k
\end{equation}
and define
\begin{equation}\label{eq:Vk}
V_k(\beta)=\int_{[0,1]^k} e(\beta\br{\textbf{t}}^d)\, d\mathbf{t}.
\end{equation}
It is clear that
$v(\beta)=\br{\bigX}V_k(\br{\bigX}^d\beta).
$
For $\boldsymbol{\beta}=(\beta_1,\ldots,\beta_R)\in \R^R$, let
\begin{equation}\label{eq:Vkprod}
 \mathfrak{V}_{\Lambda}(\boldsymbol{\beta})=\prod_{j=1}^s V_k(L_j(\boldsymbol{\beta})).
\end{equation}
We define the singular integral as
\begin{equation}\label{Rsingint}
 \mathfrak{I}^{+}(\Lambda)=\mathfrak{I}_k^{+}(\Lambda)=\int_{\R^R}\mathfrak{V}_{\Lambda}(\boldsymbol{\beta}) \, \dd\boldsymbol{\beta},
\end{equation}
where $\dd\boldsymbol{\beta}=\dd\beta_1\cdots \dd\beta_R$. Also, define its truncation as
\begin{equation}\label{Rsinginttrun}
 \mathfrak{I}^{+}(\Lambda,Y)=\int_{[-Y,Y]^R} \mathfrak{V}_{\Lambda}(\boldsymbol{\beta}) \, \dd\boldsymbol{\beta}.
\end{equation}
In the following section, we verify the convergence of these quantities.

\subsection{Convergence of singular series and singular integral}
We prove the following lemma using arguments from \cite[Lemma 35]{DaLe1966} and \cite[Lemma 23]{DaLe1969}.

\begin{lem}\label{lem:singser}
The singular series $\mathfrak{S}(\Lambda)$ converges absolutely for $s\geq R(n_0+1)$. We have
\begin{equation}
 \mathfrak{S}(\Lambda)\ll K^{\frac{n_0+1}{d}+R-1}
\end{equation}
and
\begin{equation} 
\mathfrak{S}(\Lambda,Y)=\mathfrak{S}(\Lambda)+O(Y^{-\frac{1}{2d}}K^{\frac{n_0+1}{d}+R-1}).
\end{equation}
\end{lem}
\begin{proof}
Recall \eqref{Rsingsertrun}. The lemma follows once we show that
\begin{equation}\label{Rserconvgoal} T_{\Lambda}(q)\ll K^{\frac{n_0+1}{d}+R-1}q^{-(1+\frac{1}{d})+\epsilon}.\end{equation}
 By \cite[Lemma 3.3]{BlBr2018}, we have
 \[q^{-k} S_k(q,a) \ll \tau(q)^{k-1} q^{-1/d}\]
 if $(a;q)=1$, and
 \[q^{-k} S_k(q,a)= {q'}^{-k} S_k(q',a')\]
 if $a/q=a'/q'$. Let $u_j=(q;L_j(\mathbf {A}))$ and $q_j=q/(q;L_j(\mathbf{A}))$. Using the above bounds for $S(q,L_j(\mathbf{A}))$ with $1\leq j\leq R(n_0+1)$ and using the trivial bound $\lvert q^{-k} S(q,L_j(\mathbf{A}))\rvert \leq 1$ for $R(n_0+1)<j\leq s$, we obtain
 \[T_{\Lambda}(q)\ll \sum_{\mathbf{A}}\prod_{j=1}^{R(n_0+1)} q_j^{-1/d+\epsilon},\]
 where the sum is over all $0\leq A_1,\ldots, A_R \leq q$ with $(A_1;\ldots;A_R,q)=1$. By H\"{o}lder's inequality, \eqref{Rserconvgoal} follows once we show that
\begin{equation}\label{Rserconvgoal2}\sum_{\mathbf{A}}\prod_{j=1}^{R} q_{tR+j}^{(n_0+1)(-1/d+\epsilon)}\ll K^{\frac{n_0+1}{d}+R-1}q^{-(1+\frac{1}{d})+\epsilon}\end{equation}
for any $0\leq t\leq n_0$. Without loss of generality, let $t=0$. Then
\[\sum_{\mathbf{A}}\prod_{j=1}^{R} q_{j}^{(n_0+1)(-1/d+\epsilon)}=q^{R(n_0+1)(-1/d+\epsilon)}\sum_{\mathbf{A}} (u_1\cdots u_{R})^{(n_0+1)(1/d-\epsilon)}.\]
Let $u=(u_1;\ldots;u_R)$, $\Lambda_1=[C_1,\ldots,C_R]$, $\Delta_1=\det(\Lambda_1)$, and $\Lambda_1'=\Delta_1 (\Lambda_1^{\mathrm{T}})^{-1}$. Then $\Lambda_1'$ is an integer matrix, and we have
\[\Lambda_1'\left(\begin{array}{c}
 L_1 (\mathbf{A})\\
 \vdots \\
 L_R (\mathbf{A})
\end{array}\right)=\left(\begin{array}{c}
 \Delta_1 A_1 \\
 \vdots \\
 \Delta_1 A_R
\end{array}\right).\]
Using $u\vert q$ and $(A_1;\ldots;A_R;q)=1$, it follows that $u\vert \Delta_1$. By \eqref{Kdef}, we deduce that $u\leq K$. Also, since $A_j\leq q$, we have $L_j(\mathbf{A})\leq Kq$. Since $\Lambda_1$ is invertible, there is at most one choice of $\mathbf{A}=(A_1,\ldots,A_R)$ for each choice of $L_1(\mathbf{A}),\ldots,L_R(\mathbf{A})$. Furthermore, since $u_j\mid L_j(\mathbf{A})$, for fixed $(u_1, \dots, u_R)$ with $u_j\mid q$, the number of such tuples $(L_1(\mathbf{A}),\ldots,L_R(\mathbf{A}))$ is $O \Bigl(\frac{Kq}{u_1}\cdots \frac{Kq}{u_R}\Bigr)$. Thus
\[\sum_{\mathbf{A}} (u_1\cdots u_{R})^{(n_0+1)(1/d-\epsilon)} \ll K^Rq^R \sum_{\substack{u_1,\ldots,u_R\vert q \\ (u_1;\ldots ;u_R)\leq K}} (u_1\cdots u_R)^{\frac{n_0-d+1}{d}} .\]
Write $u_j=u v_j$ with $(v_1;\ldots;v_R)=1$. Then $v_j \vert \frac{q}{u}$ for all $j$ and $v_1\cdots v_R \vert (\frac{q}{u})^{R-1}$.
Thus, for any $\epsilon>0$, we have
\[\begin{aligned}
 \sum_{\substack{u_1,\ldots,u_R\vert q \\ (u_1;\ldots ;u_R)\leq K}} (u_1\cdots u_R)^{\frac{n_0-d+1}{d}}&= \sum_{\substack{u\vert q \\ u\leq K}} \sum_{\substack{v_1,\ldots,v_R \vert \frac{q}{u} \\ (v_1;\ldots;v_R)=1}} u^{\frac{R(n_0-d+1)}{d}} (v_1\cdots v_R)^{\frac{n_0-d+1}{d}} \\
 &\ll \sum_{\substack{u\vert q \\ u\leq K}} \sum_{w\vert (\frac{q}{u})^{R-1}} u^{\frac{R(n_0-d+1)}{d}} w^{\frac{n_0-d+1}{d}+\epsilon} \\
 &\ll K^{\frac{n_0-d+1}{d}}q^{(R-1)(\frac{n_0-d+1}{d}+\epsilon)}.
\end{aligned}
\]
where $w^\epsilon$ comes from the estimation on the number of divisors of $w$. Substituting back, we conclude that
\[
\begin{aligned}
 \sum_{\mathbf{A}}\prod_{j=1}^{R} q_{j}^{(n_0+1)(-1/d+\epsilon)} &\ll q^{R(n_0+1)(-1/d+\epsilon)} K^R q^R K^{\frac{n_0-d+1}{d}}q^{(R-1)(\frac{n_0-d+1}{d}+\epsilon)} \\
 &\ll K^{R+\frac{n_0-d+1}{d}}q^{-\frac{n_0-d+1}{d}+\epsilon}.
\end{aligned}
\]
Since $n_0\geq 2d$, this implies \eqref{Rserconvgoal2}. The proof follows.
\end{proof}

\begin{lem}\label{lem:singint}
The singular integral $\mathfrak{I}^{+}(\Lambda)$ converges absolutely for $s\geq R(n_0+1)$. We have
\begin{equation}\label{singint1}\mathfrak{I}^{+}(\Lambda) \ll 1.\end{equation}
Also, for any $Y\geq K$, we have
\begin{equation}\label{singint2}\mathfrak{I}^{+}(\Lambda,Y)=\mathfrak{I}^{+}(\Lambda)+ O(KY^{-1}).\end{equation} 
\end{lem}
\begin{proof}
Applying the trivial bound $V_k(L_j(\boldsymbol{\beta}))\ll 1$ for $R(n_0+1)<j\leq s$, we obtain
 \[\mathfrak{I}^{+}(\Lambda)\ll \int_{\R^R}\prod_{j=1}^{R(n_0+1)} \abs{V_k(L_j(\boldsymbol{\beta}))} \, \dd\boldsymbol{\beta}.\]
By H\"{o}lder's inequality, \eqref{singint1} follows once we show that
\[\int_{\R^R}\prod_{j=1}^{R} \abs{V_k(L_{Rt+j}(\boldsymbol{\beta}))}^{n_0+1} \, \dd\boldsymbol{\beta} \ll 1\]
for any $0\leq t\leq n_0$. Without loss of generality, let $t=0$. Let $\Lambda_1=[C_1,\ldots,C_R]$ be the leftmost $R\times R$ submatrix of $\Lambda$. By a linear change of variables $\gamma_j=L_j(\boldsymbol{\beta})$, we obtain
\[\int_{\R^R}\prod_{j=1}^{R} \abs{V_k(L_{j}(\boldsymbol{\beta}))}^{n_0+1} \, \dd\boldsymbol{\beta} 
=\frac{1}{\Delta_1} \int_{\R^R} \abs{V_k(\gamma_1)\cdots V_k(\gamma_R)}^{n_0+1} \, \dd\boldsymbol{\gamma}
\leq \Bigl(\int_{\R} \abs{V_k(\beta)}^{n_0+1} \,d\beta \Bigr)^{R}, \]
where $\Delta_1=|\det \Lambda_1|\geq 1$. By \cite[(3.11)]{BlBr2018}, we have
\begin{equation}\label{Vkineq}\abs{V_k(\beta)}^{n_0+1}\ll \abs{\beta}^{-\frac{n_0+1}{d}}(1+\log\abs{\beta})^{(k-1)(n_0+1)}\end{equation}
for $\abs{\beta}\geq 1$. Using the trivial bound for $|\beta|<1$, \eqref{singint1} follows from \eqref{Vkineq} and the fact that $n_0\geq 2d$.

For \eqref{singint2}, again by H\"{o}lder's inequality, it suffices to show that
\[\int_{\R^R\setminus [-Y,Y]^R}\prod_{j=1}^{R} \abs{V_k(L_{j}(\boldsymbol{\beta}))}^{n_0+1} \, \dd\boldsymbol{\beta}\ll KY^{-1}.\]
After the linear change of variables by $\Lambda_1$ as before, the integral becomes
\[
\frac{1}{\Delta_1} \int_{M} \abs{V_k(\gamma_1)\cdots V_k(\gamma_R)}^{n_0+1} \, \dd\boldsymbol{\gamma},
\]
where $M$ is the image of $\R^R\setminus [-Y,Y]^R$ under this linear transformation. We claim that $M$ is contained in the region 
\[\{(\gamma_1, \ldots, \gamma_R) \in \R^R: \max{\abs{\gamma_j}}\geq K^{-1}Y\}.\]
It suffices to show for any $(y_1,\ldots,y_R) \in [-K^{-1}Y,K^{-1}Y]^R$ we have 
\[
(z_1,\ldots,z_R)^{\mathrm{T}}=\Lambda_1^{-1} (y_1,\ldots,y_R)^{\mathrm{T}}\in [-Y,Y]^R.
\]
Take $\mu_{i,j}\in \Z$ satisfying 
$\Lambda_1^{-1}=\frac{1}{\Delta_1}(\mu_{i,j})_{1\leq i,j\leq R}.$ 
By \eqref{Kdef}, we have $\abs{\mu_{i,j}}\leq \frac{K}{R}$, thus
\[\max\abs{z_j}\leq \frac{1}{\Delta_1}\cdot R\cdot \max\abs{\mu_{i,j}}\cdot \frac{Y}{K}\leq Y.\]
 This proves our claim. From this, we deduce that
\[\begin{aligned}
\frac{1}{\Delta_1} \int_{M} \abs{V_k(\gamma_1)\cdots V_k(\gamma_R)}^{n_0+1} \, \dd\boldsymbol{\gamma} &\leq R\int_{\R^{R-1}}\int_{\R\setminus[-Y/K,Y/K]}\abs{V_k(\gamma_1)\cdots V_k(\gamma_R)}^{n_0+1} \,d\gamma_1\cdots d\gamma_R \\
&\ll \int_{\R\setminus[-Y/K,Y/K]}\abs{V_k(\gamma_1)}^{n_0+1} \,d\gamma_1.\end{aligned}\]
Since $YK^{-1}\geq 1$, applying \eqref{Vkineq} to the above yields \eqref{singint2}.
\end{proof}
\begin{cor}\label{cor:truncprod}For $s\geq R(n_0+1)$, we have
\[\mathfrak{S}(\Lambda,W)\mathfrak{I}^{+}(\Lambda,W)=\mathfrak{S}(\Lambda)\mathfrak{I}^{+}(\Lambda)+O(W^{-\frac{1}{2d}}K^{\frac{n_0+1}{d}+R}).\]
\end{cor}
\noindent Recall from \eqref{eq:omegaRUdef} and \eqref{eq:Wdef} that $WK^{-1}=U^s\geq 1$. The proof is straightforward.

\subsection{Positivity of singular series and singular integral}
In this section, we give a sequence of lemmas leading to the positivity of singular series and integrals. Recall that 
\[T_{\Lambda}(q)=q^{-ks}\sum_{\substack{0\leq A_1,\ldots,A_R\leq q \\ (A_1;\ldots;A_R;q)=1}} \prod_{j=1}^s S_k(q,L_j(\mathbf{A})).\]
 \begin{lem}\label{multipicative-sieres}The arithmetic function
$T_{\Lambda}(q)$
is multiplicative in $q.$
 \end{lem}
\begin{proof}
The proof is standard. For example, it is analogous to the proof of \cite[Lemma 5.1]{DaBook2005}. 
\end{proof}

 By multiplicativity of $T_{{\Lambda}}(q)$, we have
 $
 \mathfrak{S}(\Lambda)=\prod_{p} E_p(\Lambda),
 $
 where
 $
 E_p(\Lambda)=\sum_{l\geq 0} T_{\Lambda}(p^l).
 $
 Let $\Phi_{\Lambda}(q)$ denote the number of solutions to the system of congruence equations 
 \begin{equation}\label{congrelq}
 \sum\limits_{j=1}^s \lambda_{1,j} \langle \bx_j \rangle^d \equiv \sum\limits_{j=1}^s\lambda_{2,j} \langle \bx_j \rangle^d \equiv\cdots\equiv\sum\limits_{j=1}^s \lambda_{R,j}\langle \bx_j \rangle^d \equiv 0\pmod q.
 \end{equation}
 The proof of the following lemma is standard, but we include it for self-containment. 
 \begin{lem}
For any $L\geq 1$, we have
 \[
 \sum_{0\leq l\leq L}T_{\Lambda}(p^l)=\frac{\Phi_{\Lambda}(p^L)}{p^{L(ks-R)}}.
 \]
 \end{lem}
 \begin{proof}
For $1\leq j\leq s$, write $\bx_j=(x_{1,j},\ldots,x_{k,j})$. By a standard argument, for example in \cite[Proof of Lemma 5.3]{DaBook2005}, we can write $\Phi_{\Lambda}(q)$ as an exponential sum
\begin{align*}
\Phi_{\Lambda}(q)
& = \frac{1}{q^R}\sum_{\substack{q_1\mid q}}\sum_{\substack{1\leq t_1, \ldots, t_R \leq q_1\\ (t_1; \ldots; t_R; q_1)=1}}\sum_{\substack{1\leq x_{i, j}\leq q\\ 1\leq i\leq k, \, 1\leq j\leq s}} e\left(\frac{\sum_{1\leq j\leq s} L_j(\bt) \langle \bx_j \rangle^d}{q_1}\right).
\end{align*}
By periodicity and the definition of $T_\Lambda(q)$, we see that
\[\begin{aligned}\Phi_{\Lambda}(q)
& = \frac{1}{q^R}\sum_{\substack{q_1\mid q}}\sum_{\substack{1\leq t_1, \ldots, t_R \leq q_1\\ (t_1; \ldots; t_R; q_1)=1}}\sum_{\substack{1\leq x_{i, j}\leq q_1 \\ 1\leq i\leq k, \, 1\leq j\leq s}} \left(\frac{q}{q_1}\right)^{ks}e\left(\frac{\sum_{1\leq j\leq s} L_j(\bt) \langle \bx_j \rangle^d}{q_1}\right) \\ 
&=q^{ks-R}\sum_{q_1\mid q} T_{\Lambda}(q_1). \end{aligned}\]
Letting $q=p^L$, the lemma follows.
 \end{proof}

\begin{prop}\label{prop:series}
 If the system of equations
 \begin{equation}\label{eq:series}
 \sum\limits_{j=1}^s \lambda_{1,j} y_j^d = \sum\limits_{j=1}^s\lambda_{2,j} y_j^d =\cdots=\sum\limits_{j=1}^s \lambda_{R,j}y
 _j^d =0
 \end{equation}
has solutions in $\Q_p$ with $y_j\neq 0$, $1\leq j\leq s$ for every $p$, then $\mathfrak{S}(\Lambda)>0$.
\end{prop} 

\begin{proof}
Observe that $E_p(\Lambda)=\lim\limits_{L\to \infty} \frac{\Phi_{\Lambda}(p^L)}{p^{L(ks-R)}}$. Since the singular series is absolutely convergent and satisfies $\mathfrak{S}(\Lambda)=\prod_{p} E_p(\Lambda)$, it suffices to prove that for every prime number $p$ there exists a positive number $c_p$ such that
\[
\frac{\Phi_{\Lambda}(p^L)}{p^{L(ks-R)}} \geq c_p \quad \text{for all} \quad L\gg 1.
\]
Since the equation \eqref{eq:series} is homogeneous, we may assume there exists a solution $\alpha_1, \ldots, \alpha_s\in \Z_p$ to \eqref{eq:series} with all $\alpha_j\neq 0$. It follows from a multivariable version of Hensel's Lemma \cite[Theorem 25]{Ku2011} that there exists a $\gamma>0$ such that for any choice of $\alpha'_{R+1},\ldots, \alpha'_s\in \Z_p$ with $\alpha'_i\equiv \alpha_i \pmod {p^{\gamma}}$ there exists a $\alpha'_1,\ldots,\alpha'_R \in \Z_p$ such that $(\alpha'_1,\ldots, \alpha'_s)$ is a solution to \eqref{eq:series}.

Let $z_1, \ldots, z_s\in \Nb$ satisfy $z_j\equiv y_j\pmod {p^\gamma}$, and let $\bz_j=(z_j, 1,\ldots, 1)\in \Nb^k$. Observe that $\langle \bz_j \rangle=z_j$ and 
\[
 \sum\limits_{j=1}^s \lambda_{1,j} \langle \bz_j \rangle^d \equiv \sum\limits_{j=1}^s\lambda_{2,j} \langle \bz_j \rangle^d \equiv\cdots\equiv\sum\limits_{j=1}^s \lambda_{R,j}\langle \bz_j \rangle^d \equiv 0\pmod {p^\gamma}. 
 \]
Let $L>\gamma$, and take $\bx_{R+1}, \ldots, \bx_s \in (\Z/p^L \Z)^{k}$ such that $\bx_j\equiv \bz_j\pmod{{p^\gamma}}$ for all $R+1\leq j\leq s$. There are $p^{(L-\gamma)k(s-R)}$-many choices for $\bx_{R+1}, \ldots, \bx_s$. By construction, we have the relations
 \[
 \sum_{1\leq j\leq R}\lambda_{i,j} z_j^d +\sum_{R+1\leq j\leq s} \lambda_{i,j} \langle \bx_j \rangle^d \equiv 0 \pmod {p^{\gamma}}, \quad 1\leq i\leq R.
 \]
For each $R+1\leq j\leq s$,  we take a lift of $\br{\bx_j}$ to $\Z_p$, and continue to denote the lift by $\br{\bx_j}$. Then there exists $z'_1,\ldots,z'_R\in\Z_p$ such that
 \[ \sum_{1\leq j\leq R}\lambda_{i,j} {z'}_j^d +\sum_{R+1\leq j\leq s} \lambda_{i,j} \langle \bx_j \rangle^d =0 , \quad 1\leq i\leq R.
 \] 
Therefore, any $\bx_1,\ldots, \bx_R\in (\Z/p^L\Z)^k$ that satisfies $\ang{\bx_j}\equiv z'_j \pmod {p^L}$ is a solution to
\[
 \sum\limits_{j=1}^s \lambda_{1,j} \langle \bx_j \rangle^d \equiv \sum\limits_{j=1}^s\lambda_{2,j} \langle \bx_j \rangle^d \equiv\cdots\equiv\sum\limits_{j=1}^s \lambda_{R,j}\langle \bx_j \rangle^d \equiv 0\pmod {p^L}.
\]
Thus $\Phi_{\Lambda}(p^L) \geq \varphi(p^L)^{R(k-1)} p^{(L-\gamma)k(s-R)} $, and from this we immediately see that 
\[\frac{\Phi_{\Lambda}(p^L)}{p^{L(ks-R)}} \geq p^{-\gamma k(s-R)-(k-1)R}.\]
This finishes the proof of the lemma.
\end{proof}
Now, we verify the positivity of the singular series. Recall \eqref{eq:Vk}. As in \cite[Proof of Lemma 4.3]{BlBr2018}, we apply a change of variables to $V_k(\beta)$. Let $\bt' = (t_2, \ldots, t_k)$, and write $v=t_1^d\ldots t_k^d$ so that
\[
V_k(\beta)=\int_{[0, 1]^{k-1}} \ang{\bt'}^{-1}\int_{0}^{\ang{\bt'}^d} d^{-1}v^{(1/d)-1}e(\beta v)\ \dd v \ \dd \bt'.
\]
 If we define
\[
\mathscr{U}(v)=\{\bt'\in [0, 1]^{k-1}: \ang{\bt'}^d\geq v\} \quad \text{and}\quad \psi_k(v)=v^{(1/d)-1} \int_{\mathscr{U}(v)} \ang{\bt'}^{-1} \dd \bt',
\]
we arrive at the expression
\begin{equation}\label{eq:Vknew}
 V_k(\beta)=\frac{1}{d}\int_0^1 \psi_k(v)e(\beta v)\dd v.
\end{equation}
Observe that the change in the order of integration is justified by Tonelli's theorem.
\begin{prop}\label{prop:singint}
 If the equation \eqref{eq:series} has positive real solutions $y_1, \ldots, y_s$, then
$
\mathfrak{I}^{+}(\Lambda)>0.
$
\end{prop}
\begin{proof}
Recall \eqref{eq:Vkprod} and \eqref{Rsingint}. By \eqref{eq:Vknew} we have
\begin{align*}
 \mathfrak{I}^{+}(\Lambda)& =d^{-s}\int_{\R^R}\int_{[0, 1]^s}\psi_k(v_1)\ldots \psi_k( v_s)e\Bigl(\sum_{1\leq j\leq s} L_j(\boldsymbol{\beta})v_j\Bigr) \, \dd \bv \dd \boldsymbol{\beta}.
 \end{align*}
Observe that we can write the linear form inside the complex exponential in the above integral as $\sum\limits_{1\leq i\leq R} R_i(\bv)\beta_i$, where $
 R_i(\bv)= \sum\limits_{1\leq j\leq s}\lambda_{i,j} v_j$. Let $\Lambda_1=[C_1, \ldots , C_R]$. By the assumption written after Lemma~\ref{lem:Aigner}, $\Lambda_1$ is invertible. Apply the change of variables from $v_1,\ldots,v_R$ to $u_1,\ldots,u_R$, where
$u_i=R_i(\bv)$, and leave $v_{R+1},\ldots,v_s$ unchanged. Write $\bu=(u_1,\ldots,u_R)$, $\bv'=(v_{R+1},\ldots,v_s)$ and let $L'_i(\bu,\bv')$, $1\leq i\leq R$ be defined by
\[\left(\begin{array}{c} L'_1(\bu,\bv') \\ \vdots \\ L'_R(\bu,\bv') \end{array}\right)= \Lambda_1^{-1} \left(\begin{array}{c} u_1 - \sum\limits_{R+1\leq j\leq s} \lambda_{1,j} v_j\\ \vdots \\ u_R - \sum\limits_{R+1\leq j\leq s} \lambda_{R,j} v_j\end{array}\right).\]
Note that $L'_i(\bu,\bv')=v_i$ for all $1\leq i\leq R$. We deduce that
\[\mathfrak{I}^{+}(\Lambda)=d^{-s}{\Delta_1}^{-1}\int_{\R^R}\int_{\R^R}B(\bu)e\Bigl(\sum_{1\leq i\leq R} u_i \beta_i\Bigr)d\bu \dd\boldsymbol{\beta},\]
where
\begin{equation}\label{eq:Bu}B(\bu)= \int_{\mathscr{B}(\bu)}\psi_k(L'_1(\bu,\bv'))\cdots \psi_k(L'_R(\bu,\bv')) \psi_k(v_{R+1})\cdots \psi_k(v_s) d\bv' \end{equation}
with
\[\mathscr{B}(\bu)=\left\{ \bv'\in [0,1]^{s-R}: 0\leq L'_i(\bu,\bv')\leq 1 \text{ for all } 1\leq i\leq R\right\}.\]
Applying the Fourier inversion theorem $R$ times, we obtain
\[\mathfrak{I}^{+}(\Lambda)=d^{-s}{\Delta_1}^{-1}B(0,\ldots,0).\]
The existence of positive solutions implies that $\mathscr{B}(0,\ldots,0)$ contains a $(s-R)$-dimensional box of positive volume. Since the integrand in \eqref{eq:Bu} is positive, the proposition follows.
\end{proof}

\section{Asymptotic formulas for box sums}
In this section, we give the asymptotic formulas for the number of positive solutions, all solutions, and primitive solutions of \eqref{eq:equation}, under the hypotheses given in \Cref{thm:boxsum}. 
\subsection{Asymptotic formula for positive solutions}
We first give an asymptotic formula for $M_{\Lambda}^{+}(\bigX)$ under the assumption \eqref{eq:Psize}, which guarantees the disjointness of the major arcs, and then proceed to show that the formula remains valid even when \eqref{eq:Psize} does not hold. Recall \eqref{eq:ML+int}, and that we have $M_\Lambda^+(\bigX)=I(\mathfrak{K})+I(\mathfrak{k})$. Let $E(q,\bigX)$ be the symmetric function in $X_1,\ldots, X_k$ defined by
\[E(q,\bigX)=q^k+\sum_{r=1}^{k-1}q^{k-r}X_1\cdots X_r.\]
Let $\boldsymbol{\alpha}=(\alpha_1,\ldots,\alpha_R)\in \mathfrak{K}$. Write $\alpha_i=\frac{A_i}{q}+\beta_i$, where $1\leq q\leq W$, $(q;A_1;\ldots;A_R)=1$, and $|\beta_i|\leq WP^{-d}$ for all $1\leq i\leq R$. Write $\mathbf{A}=(A_1,\ldots,A_R)$. By \cite[Lemma 3.2]{BlBr2018}, for $1\leq j\leq s$, we have
\[\begin{aligned}
 f(L_j(\balpha))&=f\Bigl(\frac{L_j(\mathbf{A})}{q}+L_j(\bbeta)\Bigr) \\
 &=q^{-k} S(q,L_j(\mathbf{A}))v(L_j(\bbeta))+O\bigl(E(q,\bigX)K^kW^k\bigr),
\end{aligned}\]
where $v(L_j(\boldsymbol{\beta}))$ is as defined in \eqref{eq:vk}.
By the trivial upper bounds $q^{-k} S(q,L_j(\mathbf{A}))v(L_j(\bbeta))\ll P$ and $f(L_j(\balpha))\ll P$, we deduce
\[F(\balpha)=q^{-sk}\prod_{j=1}^s S(q,L_j(\mathbf{A}))v(L_j(\bbeta))+O\bigl(P^{s-1} E(q,\bigX)K^kW^k\bigr).\]
Observe that $E(q,\bigX)\ll W^k P X_k^{-1}$ whenever $q\leq W$, and that the measure of $\mathfrak{K}$ is $O(W^{2R+1} P^{-Rd})$. Upon recalling \eqref{def:majorarc}, \eqref{def:IKk}, and \eqref{Rsingsertrun}, we obtain
\[\begin{aligned}I(\mathfrak{K})&=\mathfrak{S}(\Lambda,W)\int_{[-WP^{-d},WP^{-d}]^R} \prod_{j=1}^s v(L_j(\bbeta)) \, d\bbeta + O(K^k W^{2R+2k+1}X_k^{-1}P^{s-Rd}) \\
&=\mathfrak{S}(\Lambda,W)\mathfrak{I}^{+}(\Lambda,W)P^{s-Rd}+O\Bigl(K^{2R+3k+1} U^{(2R+2k+1)s}X_k^{-1}P^{s-Rd}\Bigr).
\end{aligned}\]
Recall from \eqref{eq:omegaRUdef} that $U^s=X_k^{k\omega_R}$. By Corollary~\ref{cor:truncprod}, we deduce
\[I(\mathfrak{K})=\mathfrak{S}(\Lambda)\mathfrak{I}^{+}(\Lambda)P^{s-Rd}+O\Bigl(K^{n_0+1+3k+2R} P^{s-Rd} X_k^{\sigma}\Bigr),\]
where
\[\sigma=\max\left(-\frac{k\omega_R}{2d},k\omega_R(2R+2k+1)-1\right).\]
Since
\[0<k\omega_R(2R+2k+1)= \frac{k(2R+2k+1)}{(8dk)^{8}R}\leq \frac{5}{2^{24}d^8k^6}<\frac{1}{2^{21}},\]
we have $-1<\sigma<0$. Therefore, together with Corollary~\ref{cor:Rminor},
we conclude that there exists $0<\delta<1$ such that
\[
M_{\Lambda}^{+}(\bigX)=I(\mathfrak{K})+I(\mathfrak{k})=\mathfrak{S}(\Lambda)\mathfrak{I}^{+}(\Lambda)P^{s-Rd}+O(K^{n_0+1+3k+2R} P^{s-Rd} X_k^{-\delta}).\]
Next, assume that \eqref{eq:Psize} does not hold. In this case, we have $P\ll K^{\frac{3}{d}(1-\frac{3}{2^{24}})^{-1}},$ 
thus 
\[K^{3(1-\frac{3}{2^{24}})^{-1}(R+1)}P^{s-Rd}X_k^{-\delta}\gg K^{3(1-\frac{3}{2^{24}})^{-1}(R+1)} P^{s-(R+1)d} \gg P^s.\]
Together with Lemma~\ref{lem:singser} and Lemma~\ref{lem:singint}, we conclude that for \begin{equation}\label{Adef}A=\max\Bigl(n_0+1+3k+2R,3(1-\frac{3}{2^{24}})^{-1}(R+1)\Bigr),\end{equation}
we have 
\begin{equation}\label{Rposasymp}
M_{\Lambda}^{+}(\bigX)=\mathfrak{S}(\Lambda)\mathfrak{I}^{+}(\Lambda)P^{s-Rd}+O(K^{A} P^{s-Rd} X_k^{-\delta}).\end{equation}
This completes the proof.
\subsection{Asymptotic formula for all solutions}
We now complete the proof of Theorem~\ref{thm:boxsum}. Let $C^+_\Lambda=\mathfrak{S}(\Lambda)\mathfrak{I}^{+}(\Lambda)$. By \eqref{Rposasymp}, we have
\[
M^+_\Lambda(\bigX)=C^+_\Lambda P^{s-Rd}+E^+_\Lambda
\]
where $E^+_\Lambda\ll K^{A} P^{s-Rd} X_k^{-\delta}$ with $A$ as defined in \eqref{Adef}. For $\boldsymbol{\eta}=(\eta_j)_{ \substack{ 1\leq j\leq s}}$, let $\boldsymbol{\eta}\Lambda$ denote the matrix
 \[
 \boldsymbol{\eta}\Lambda=(\eta_j\lambda_{i,j})_{\substack{1\leq i\leq R\\ 1\leq j\leq s}}.
 \] 
We can apply the combinatorial arguments from \cite[Section 4.4]{BlBr2018} with minor modifications to deduce that
\[ M_{\Lambda}(\bigX)=2^{ks}M^+_{\Lambda}(\bigX)\]
if $d$ is even, and 
\[ M_{\Lambda}(\bigX)=2^{(k-1)s}\sum_{\substack{\eta_j\in \{\pm 1\} \\1\leq j\leq s
 }} M^{+}_{\boldsymbol{\eta}\Lambda}(\bigX)\]
 if $d$ is odd.
Therefore, we can write
\[ M_{\Lambda}(\bigX)=C_\Lambda P^{s-Rd}+E_\Lambda,\]
where $C_\Lambda$ equals either \[2^{ks}C^+_\Lambda \quad \text{or} \quad 2^{(k-1)s} \mathfrak{S}(\Lambda)\sum_{\substack{\eta_j\in \{\pm 1\}\\1\leq j\leq s}} \mathfrak{I}^+(\boldsymbol{\eta}
 \Lambda)\] depending on whether $d$ is even or odd, respectively, and
$E_\Lambda\ll K^{A} P^{s-Rd} X_k^{-\delta}$. From this, we immediately see that $C_{\Lambda}>0$ if the system of equations \eqref{eq:series} has solutions in $\R$ and in $\Q_p$ for all primes $p$, with $y_j\neq 0$ for all $1\leq j\leq s$. Finally, to remove the dependence on the parity of $d$, let
\begin{equation}\label{eq:Sintegral}
 \mathfrak{I}(\Lambda)= \int_{\R^R}\int_{[-1, 1]^{ks}} e\Bigl(L_1(\boldsymbol{\beta})\br{\bt_1}^d+\cdots +L_s(\boldsymbol{\beta})\br{\bt_s}^d\Bigr) d\bt \dd\boldsymbol{\beta},
 \end{equation}
where $\bt_j=(t_{1,j},\ldots,t_{k,j})$. A direct computation shows that
$\mathfrak{I}(\Lambda)$ equals $2^{ks}\mathfrak{I}^+(\Lambda)$ if $d$ is even, and equals \[2^{(k-1)s} \sum\limits_{\substack{\eta_j\in \{\pm 1\}\\1\leq j\leq s}}\mathfrak{I}^+(\boldsymbol{\eta}
 \Lambda)\] if $d$ is odd.
We conclude that \eqref{eq:boxsum} holds with
\begin{equation}\label{eq:boxsumconst}
 C_{\Lambda}=\mathfrak{S}(\Lambda)\mathfrak{I}(\Lambda),
 \end{equation}
 with $\mathfrak{S}(\Lambda)$ and $\mathfrak{I}(\Lambda)$ given by \eqref{Rsingser} and \eqref{eq:Sintegral}, respectively.

\subsection{Asymptotic formula for primitive solutions}\label{sec:priasymp}
Let $M_\Lambda^*(\bigX)$ denote the number of integral solutions to \eqref{eq:equation} satisfying \eqref{eq:boxbound}, with the additional constraint
\[(x_{i,1};\ldots;x_{i,s})=1, \quad 1\leq i\leq k.\]
Following the treatment from \cite[Section 4.5]{BlBr2018} with minor modifications, it is straightforward to check that there exists $\delta>0$ such that 
\begin{equation}\label{eq:priasymp} 
M^*_{\Lambda}(\mathbf{X})=\zeta(s - Rd)^{-k} C_\Lambda\br{\bigX}^{s-Rd} +O\bigl(K^{A} \br{\bigX}^{s-Rd} (\min X_i)^{-\delta}\bigr),
\end{equation}
where $A$ is the number defined in \eqref{Adef}.

\section{Asymptotic formula for hyperbolic sum}
To conclude this paper, we deduce the asymptotic formula for $N(B)$ and finish the proof of Theorem~\ref{thm:hypsum}. For $\boldsymbol{m}=(m_1, \ldots, m_k)\in \mathbb{N}^k$, define
\[
\theta_\Lambda(\boldsymbol{m})=\#\{(\mathbf{x}_1, \ldots, \mathbf{x}_k)\in X_0(\Q): \text{ $\mathbf{x}_i\in \Z^s$ primitive, $\abs{\mathbf{x}_i}=m_i$, $x_{i, j} \neq 0$ \ $\forall \, 1\leq i\leq k$, $1\leq j\leq s$} \}.
\]
 Recall \eqref{eq:upsilon}. We have
\begin{equation}\label{eq:NBdef}
N(B) = \frac{1}{2^k} \sum_{m_1 m_2 \cdots m_k \leq B^{1/(s-Rd)}} \theta_\Lambda(\mathbf{m})=\frac{1}{2^k} \Upsilon(B^{1/(s-Rd)}).
\end{equation}
To apply Theorem~\ref{thm:BlBrhyp}, we shall show that $\theta_\Lambda(\mathbf{m})$ belongs to the $(\alpha, c, D, \nu, \delta)$-family of arithmetic functions, described in Section~\ref{sec:hypmethod}, for a suitable choice of parameters. Observe that for any $\bigX\in \mathbb{N}^k$ we have
\[\sum_{\mathbf{m}\leq \bigX} \theta_\Lambda(\mathbf{m})=M_\Lambda^*(\bigX).\]
By \eqref{eq:priasymp}, we see that $\theta_\Lambda(\mathbf{m})$ satisfies the condition (I) with $\alpha=s-Rd$, $c=\zeta(s - Rd)^{-k} C_\Lambda$, and for some $0<\delta<1$. \par
Next, we check condition (II). Fix $r$ with $1\leq r\leq k-1$ and put $l=k-r$. For $\bu\in \Nb^r$ and $\bigV\in [1,\infty)^l$, let
\begin{equation}\label{eq:Thetauv} \Theta_{\bu}(\bigV)=\sum_{\bv\leq \bigV} \theta_\Lambda(\bu,\bv).
\end{equation}
Write \eqref{eq:equation} as
 \begin{equation}\label{eq:equationyz}
 \sum\limits_{j=1}^s \lambda_{i,j} (y_{1,j}\cdots y_{r,j}z_{1,j}\cdots z_{l,j})^d=0, \quad 1\leq i \leq R.
 \end{equation}
Observe that $\theta_\Lambda(\bu,\bv)$ is the number of solutions to \eqref{eq:equationyz} in primitive vectors $\mathbf{y}_i, \mathbf{z}_{i'} \in \mathbb{Z}^s$ satisfying $\abs{\by_i}=u_i$ and $\abs{\bz_{i'}}=v_{i'}$ for all $1\leq i\leq r$ and $1\leq i'\leq l$. To evaluate the sum \eqref{eq:Thetauv}, we sum over $\bv\leq \bigV$ for each $r$-tuple $(\by_1,\ldots,\by_r)$ in a permissible set defined as follows:
\[\mathscr{Y}(\bu)=\{(\by_1,\ldots,\by_r): \by_i\in \Z^s \text{ primitive, } \abs{\by_i}=u_i, \, y_{i, j}\neq 0  \text{ for all }1\leq i\leq r, \, 1\leq j\leq s\}.\]
For $(\by_1,\ldots,\by_r)\in \mathscr{Y}(\bu)$, let
\begin{equation}
\label{eq:Lambday}
\Lambda(\by_1,\ldots,\by_r) =\left(\begin{array}{cccc}
 \lambda_{1,1} (y_{1,1}\cdots y_{r,1})^d & \lambda_{1,2} (y_{1,2}\cdots y_{r,2})^d & \cdots &\lambda_{1,s} (y_{1,s}\cdots y_{r,s})^d \\
 \lambda_{2,1} (y_{1,1}\cdots y_{r,1})^d & \lambda_{2,2}(y_{1,2}\cdots y_{r,2})^d & \cdots &\lambda_{2,s} (y_{1,s}\cdots y_{r,s})^d \\
 \vdots & \vdots & \ddots& \vdots \\
 \lambda_{R,1} (y_{1,1}\cdots y_{r,1})^d & \lambda_{R,2} (y_{1,2}\cdots y_{r,2})^d & \cdots & \lambda_{R,s} (y_{1,s}\cdots y_{r,s})^d
\end{array}\right).
\end{equation}
We find that
\[\Theta_{\bu}(\bigV)=\sum\limits_{(\by_1,\ldots,\by_r)\in\mathscr{Y}(\bu)}M^*_{\Lambda(\by_1,\ldots,\by_r)}(\bigV).\]
Let $K(\by_1,\ldots,\by_r)$ be the value defined by \eqref{Kdef} associated to the matrix $\Lambda(\by_1,\ldots,\by_r)$. Observe that $\Lambda(\by_1,\ldots,\by_r)$ is obtained from $\Lambda$ simply by multiplying $(y_{1,j}\cdots y_{r,j})^d$ to each $j$-th column. Hence it straightforward to see that $\Lambda(\by_1,\ldots,\by_r)$ satisfies the rank hypothesis in Theorem~\ref{thm:boxsum}, and that $K(\by_1,\ldots,\by_r)$ satisfies
\begin{equation}\label{Ky} K(\by_1,\ldots,\by_r)\ll \br{\bu}^{Rd}K,\end{equation}
because 
\[K(\by_1,\ldots,\by_r)\leq\max(\br{\bu}^{Rd}\Delta, R\br{\bu}^{d}m_1, R\br{\bu}^{(R-1)d}m_2).\]
Recall the definition \eqref{Adef} of $A$. By \eqref{eq:priasymp} we have
\[\Theta_{\bu}(\bigV)=\zeta(s - Rd)^{-l} \br{\bigV}^{s-Rd} \sum\limits_{(\by_1,\ldots,\by_r)\in\mathscr{Y}(\bu)} C_{\Lambda(\by_1,\ldots,\by_r)}+E, \]
where
\[E\ll \br{\bigV}^{s-Rd} (\min V_i)^{-\delta} \sum\limits_{(\by_1,\ldots,\by_r)\in\mathscr{Y}(\bu)} (K(\by_1,\ldots,\by_r))^A. \]
The number of elements of $\mathscr{Y}(\bu)$ is $O(\br{\bu}^{s-1})$.
 By \eqref{Ky}, we deduce that
\[E\ll \br{\bigV}^{s-Rd} (\min V_i)^{-\delta} K^{A} \br{\bu}^{s-1+RdA}.\]
Since $\br{\bu}\ll \abs{\bu}^r$ and $1\leq r\leq k-1$, we conclude that $\theta_\Lambda(\mathbf{m})$ satisfies condition (II) with $\alpha$, $c$ same as before, 
\[
c_r(\mathbf{u}):= \zeta(s - Rd)^{-l} \sum\limits_{(\by_1,\ldots,\by_r)\in\mathscr{Y}(\bu)} C_{\Lambda(\by_1,\ldots,\by_r),}
\] $\nu=1$, 
$D=(k-1)(s-1)+(k-1)RdA$, 
and $\delta$ a sufficiently small positive number. 
Finally, it is straightforward to verify that $\theta_{\Lambda}(\mathbf{m})$ satisfies condition (III), by the symmetry of the equation \eqref{eq:equation}. We conclude that $\theta_{\Lambda}(\mathbf{m})$ belongs to the $(s-Rd,\zeta(s - Rd)^{-k} C_\Lambda,D,1,\delta)$-family, where $D$ is as defined above and $\delta$ is a small positive number. Therefore, the conclusion of Theorem~\ref{thm:hypsum} follows from Theorem~\ref{thm:BlBrhyp} and \eqref{eq:NBdef}, with
\[C=\frac{C_\Lambda}{2^k (k-1)!\zeta(s-Rd)^k}.\]
 Furthermore, from \eqref{eq:boxsumconst}, we see that $C$ is positive if and only if \eqref{eq:equation} has solutions in $\R$ and in $\Q_p$ for every prime $p$, with all $x_{i,j}\neq 0$. This completes the proof of Theorem~\ref{thm:hypsum}.
\subsection{A remark on uniformity}\label{sec:uniformity}
We end this paper with a quick proof of the uniformity statement given in Remark~\ref{rem:uniformity}. Let $K_0>0$ and let $\mathscr{H}_{K_0}$ be the family of systems of equations whose coefficient matrices $\Lambda$ satisfy $K(\Lambda)\leq K_0$. Then it follows from Lemma~\ref{lem:singser}, Lemma~\ref{lem:singint}, and \eqref{eq:boxsumconst} that there exists a positive constant $B>0$ such that the inequality $C_\Lambda \leq B K_0^{\frac{n_0+1}{d}+R-1}$ holds for every such $\Lambda$. We deduce that if $\Lambda\in \mathscr{H}_{K_0}$ then $\theta_{\Lambda}(\mathbf{m})$ is in the $(s-Rd,B K_0^{\frac{n_0+1}{d}+R-1},D,1,\delta)$-family. The uniformity statement follows immediately, because the asymptotic formula in Theorem~\ref{thm:BlBrhyp} holds uniformly across the family of arithmetic functions.

\bibliographystyle{amsplain}
\bibliography{hyperbolic}

\end{document}